\documentclass[oneside,english]{amsart}
\usepackage[T1]{fontenc}
\usepackage[latin9]{inputenc}
\usepackage{textcomp}
\usepackage{amstext}
\usepackage{amsthm}
\usepackage{amssymb}

\makeatletter
\numberwithin{equation}{section}
\numberwithin{figure}{section}
\theoremstyle{plain}
\newtheorem{thm}{\protect\theoremname}
\theoremstyle{plain}
\newtheorem{lem}[thm]{\protect\lemmaname}
\theoremstyle{definition}
\newtheorem{defn}[thm]{\protect\definitionname}
\theoremstyle{remark}
\newtheorem{rem}[thm]{\protect\remarkname}

\makeatother

\usepackage{babel}
\providecommand{\definitionname}{Definition}
\providecommand{\lemmaname}{Lemma}
\providecommand{\remarkname}{Remark}
\providecommand{\theoremname}{Theorem}

\begin{document}
\title{{\small{}Symmetric Solutions to Nonlinear Vectorial 2nd Order ODE's}}
\subjclass[2000]{Primary{\small{} 34A34; }Secondary{\small{} }70F15{\small{} Department
of }M{\small{}athematics , WVU, }M{\small{}organtown WV 26506}}
\email{aaabdulhussein@mix.wvu.edu{\small{}$^{1}$ }, gingold@math.wvu.edu{\small{}$^{2}$}}
\author{{\small{}Ali Abdulhussein$^{1}$ and Harry Gingold$^{2}$}}
\begin{abstract}
It is proven that second order vectorial nonlinear differential systems
$y''=f(y)$ , possess a continuum of symmetric solutions. They are
shown to possess a continuum of even solutions. If $f(y)$ is an odd
function of $y$ , then $y''=f(y)$ is shown also to possess a continuum
of odd solutions. The results apply to a significant family of second
order vectorial nonlinear differential systems that are not dissipative.
This family of differential equations includes the celebrated $N$
body problem of celestial mechanics and other central force problems.
\end{abstract}

\maketitle

\section{{\small{}Introduction}}

The purpose of this article is to study symmetric solutions to second
order nonlinear ordinary vectorial differential systems. It was motivated
by the question whether or not the celebrated equations of celestial
mechanics possess solution in which the future evolution of trajectories
of the $N$ heavenly bodies is a perfect reflection of their past.
It is shown, under appropriate conditions that nonlinear second order
ordinary vectorial differential systems $y''=f(y)$ , possess a continuum
of even solutions about any initial point $t_{0}$. Namely, for any
$t_{0}\in\mathbb{R}$ , there exist a continuum of solutions $y[(t-t_{0})]$
such that 
\begin{equation}
y[(t-t_{0})]=y([-(t-t_{0})]),\qquad(t-t_{0})\in[-L,L],\quad L>0.\label{eq:EVENSYMMETRYINTRO}
\end{equation}
If $f(y)$ is an odd function of $y$, namely if $-f(y)=f(-y)$ ,
then $y''=f(y)$ is shown to also possesses a continuum of odd solutions
about any point $t_{0}\in\mathbb{R}$. Namely, for any $t_{0}\in\mathbb{R}$
, there exist a continuum of solutions $y[(t-t_{0})]$ such that
\begin{equation}
-y[(t-t_{0})]=y[(-(t-t_{0}))],\qquad(t-t_{0})\in[-L,L],\quad L>0.\label{eq:ODDSOLUTIONABOUTt0}
\end{equation}

In contrast , non constant symmetric solutions to first order differential
systems $z'=H(z)$ are ``scarce''. We will elaborate on this scarcity
in the sequel at the end of section 3 .

In what follows we propose a result that applies to systems $y''=f(y)$
where $y$ is a column vector with $n$ scalar components. The proposed
theorem applies to a significant list of equations that include central
force problems and the pendulum equation. The existence of a continuum
of even solutions to the celebrated $N$ body problem, that is given
inhere, is a manifestation of the predictive powers of Newton's equations
of celestial mechanics. The method of successive approximations is
adapted as a main tool that yield the desired results. As a byproduct
we obtain new estimates on the interval of existence. 

We could not find this result in the voluminuous literature of ordinary
differential equations. Compare with \cite{Boyce E.W. and DiPrima R.C,BRAUERNOHEL,DIACUISENTROPIC,GuckenheimerHolmes,HILLEODESCOMPLEX,JORDAN=000026SMITH,L. Perko,P.F.Hsieh=000026Y. Sibuya}.
Neither did we come across the existence of even solutions to the
$N$ body problem. 

These results could be useful to modeling of physical phenomenon.
They could be also useful to the numerical approximations of solutions
and to phase space analysis . Given a scalar ordinary differential
equation $y''=f(y)$, let $(y,y')$ be the phase plane. Then, the
even solutions orbits intercept the $y$ axis and the odd solutions
orbits, if any, intercept the $y'$ axis.

The order of presentation is as follows. Section 2 is dedicated to
preliminary definitions and lemmas that are instrumental in proving
our main Theorem 11. We study symmetry properties of integrals and
derivatives of odd and even functions. In section 3 we specialize
the method of successive approximations to second order nonlinear
vectorial ordinary differential equations $y''=f(y)$. We prove the
main result of this article that is Theorem 11. In section 4 we list
various second order nonlinear differential systems of equations that
occur in applications and to which our main Theorem 11 applies . The
distinguished N body problem of celestial mechanics is on this list
. This means, that the N body problem possesses a continuum of celestial
mechanics solutions in which the future is a perfect reflection of
the past.

\section{Definite Integrals and Derivatives of Odd and Even Functions.}

{\small{}In preparation to proving the existence of even and odd solutions
to vectorial nonlinear differential systems of the form $y''=f(y)$
, we need some notation and preparatory lemmas. }In what follows we
may suppress the notation $(t)$ in the function notation $y(t)$
and replace it by $y$. This, when the meaning is clear. We denote
$y'(t)=\frac{dy(t)}{dt}=\frac{dy}{dt}=y'$.
\begin{lem}
Let $L>0$. Denote by $I$ an open interval $(-L,L)$ or a closed
interval $[-L,L]$.

i) Assume that $-\phi(t)=\phi(-t)$ , where $\phi(t)\in C(I)$ is
an odd function on $I$ . Then, $G(u)$ defined below

\begin{equation}
G(u):=\int_{0}^{u}\phi(s)ds,\label{eq:The Inside Integral}
\end{equation}

is an even function of $u\in I$.

ii) Assume that $\phi(t)=\phi(-t)$ , where $\phi(t)\in C(I)$ , is
an even function on $I$ . Then, $G(u)$ defined below

\begin{equation}
G(u):=\int_{0}^{u}\phi(s)ds,\label{eq:The Inside Integral-1}
\end{equation}

is an odd function on $I$.

iii) Assume that $\phi(t)\in C^{1}(I)$ is an even function. Then
the derivative $\varPsi(t):=\frac{d\phi(t)}{dt}$ is an odd function
of $t\in I$.

iv) Assume that $\phi(t)\in C^{1}(I)$ is an odd function. Then the
derivative $\varPsi(t):=\frac{d\phi(t)}{dt}$ is an even function
of $t\in I$.
\end{lem}

\begin{proof}
We first prove i) . Consider
\begin{equation}
G(-u)=\int_{0}^{-u}\phi(s)ds.
\end{equation}
Make the following change of variables
\begin{equation}
s=-v\Longrightarrow\:ds=-dv,\:and\:\{s=0\Rightarrow v=0\},\:and\:\{s=-u\Rightarrow v=u\}.\label{eq:CHANGEOFVARIABLES1}
\end{equation}

Then,
\begin{equation}
G(-u)=\int_{0}^{-u}\phi(s)ds=-\int_{0}^{u}\phi(-v)dv=\int_{0}^{u}\phi(v)dv=G(u).\label{eq:PHIEVEN}
\end{equation}

as desired. The proof for ii) is similar. Apply the change of variables
(\ref{eq:CHANGEOFVARIABLES1}) with the assumption $\phi(s)=\phi(-s)$
. Then,

\begin{equation}
G(-u)=\int_{0}^{-u}\phi(s)ds=-\int_{0}^{u}\phi(-v)dv=-\int_{0}^{u}\phi(v)dv=-G(u).
\end{equation}

as desired.

iii) We focus on the quotient below with $t\in(-L,L)$ and with $h\neq0$
and $h$ arbitrarily small.
\begin{equation}
Q(t,h):=\frac{\phi(t+h)-\phi(t)}{h}.\label{eq:Quotient1}
\end{equation}
Since $\phi(t)$ is an even function then
\begin{equation}
\phi(t+h)=\phi(-t-h),\:\phi(t)=\phi(-t).\label{eq:QUOTIENTCHNGEDBYSYMMETRY1}
\end{equation}
Substitute from (\ref{eq:QUOTIENTCHNGEDBYSYMMETRY1}) into the right
hand side of (\ref{eq:Quotient1}) to obtain 
\begin{equation}
Q(t,h)=-Q(-t,-h)=-\frac{\phi(-t-h)-\phi(-t)}{-h}.\label{eq:RELATIONQ=00003D-Q}
\end{equation}
Put $\psi(t):=\frac{d\phi(t)}{dt}$ . Take the limit as $h\rightarrow0$
in (\ref{eq:RELATIONQ=00003D-Q}) and obtain

\begin{equation}
\varPsi(t):=\frac{d\phi(t)}{dt}=lim_{h\rightarrow0}Q(t,h)=-lim_{h\rightarrow0}Q(-t,-h)=-\varPsi(-t).
\end{equation}

If $I$ is the closed interval $[-L,L],$ one must exercise caution
with signs of $h$ at the end points of $I$. Assume that $t=L.$
Then in (\ref{eq:Quotient1}) we must assume that $h<0$ in order
to make $\phi(L+h)$ well defined . Notice then that $\phi(-L-h)$
is also well defined and consequently $Q(-t,-h)$ in (\ref{eq:RELATIONQ=00003D-Q})
is also well defined so that $\frac{d\phi(L)}{dt}=-\frac{d\phi(-L)}{dt}$
exist as a one sided limit with $h<0$. An analogous argument holds
at $t=-L$ with $h>0$.

iv) We focus again on the quotient 
\begin{equation}
Q(t,h):=\frac{\phi(t+h)-\phi(t)}{h}.\label{eq:Quotient1-1}
\end{equation}
Since $\phi(t)$ is an odd function then
\begin{equation}
\phi(t+h)=-\phi(-t-h),\:\phi(t)=-\phi(-t).\label{eq:QUOTIENTCHNGEDBYSYMMETRY1-1}
\end{equation}
Substitute from (\ref{eq:QUOTIENTCHNGEDBYSYMMETRY1-1}) into the right
hand side of (\ref{eq:Quotient1-1}) to obtain 
\begin{equation}
Q(t,h)=Q(-t,-h)=\frac{\phi(-t-h)-\phi(-t)}{-h}.\label{eq:RELATIONQ=00003D-Q-1}
\end{equation}
Take the limit as $h\rightarrow0$ in (\ref{eq:RELATIONQ=00003D-Q-1})
and obtain

\begin{equation}
\varPsi(t)=lim_{h\rightarrow0}Q(t,h)=lim_{h\rightarrow0}Q(-t,-h)=\varPsi(-t).
\end{equation}

Assume that $I$ is the closed interval $[-L,L],$ and $t=L$ or $t=-L$.
Then, similar arguments as in the proof of iii) hold and the proof
is completed.
\end{proof}
\begin{lem}
Let $L>0$. Denote by $I$ an open interval $(-L,L)$ or a closed
interval $[-L,L]$. Assume that $\phi(t)\in C^{1}(I)$ . Then

a) $\phi(t)$ is an even function on $I$ iff its derivative is an
odd function on $I$ .

b) $\phi(t)$ is an odd function on $I$ iff its derivative is an
even function on $I$ .

c) Assume that $\phi(t)\in C^{0}(I)$. Let $0\leq s\leq u\leq t$
or let $t\leq u\leq s\leq0$ . Put
\begin{equation}
F(t):=\int_{0}^{t}\int_{0}^{u}\phi(s)dsdu.
\end{equation}

Then, $F(t)\in C^{2}(I)$ and $F(t)=F(-t)$ is an even function on
$I$ iff $\phi(t)\in C^{0}(I)$ and $\phi(t)=\phi(-t)$ is an even
function on $I$ .

$F(t)\in C^{2}(I)$ and $-F(t)=F(-t)$ is an odd function on $I$
iff $\phi(t)\in C^{0}(I)$ and $-\phi(t)=\phi(-t)$ is an odd function
on $I$ .
\end{lem}

\begin{proof}
The proof follows easily from Lemma 1.
\end{proof}
We clarify now what is an even and an odd function of a scalar function
$u=$$H(y_{1},y_{2},\cdots,y_{N})$ of several variables. To this
end we denote the transposed column vector $y^{T}=(y_{1},y_{2},\cdots,y_{N})$
and we put $u=$$H(y_{1},y_{2},\cdots,y_{N})=H(y)$.
\begin{defn}
Denote by $REG$ an open connected set in $\mathbb{R}^{N}$ . We say
that $H(y)$ is an even function of $y$ in $REG$ if
\begin{equation}
H(y)=H(-y),\:y\in REG.\label{eq:EVEN(VECTORy)}
\end{equation}

We say that $H(y)$ is an odd function of $y$ in $REG$ if
\end{defn}

\begin{equation}
-H(y)=H(-y),\:y\in REG.\label{eq:ODD(VECTORy)}
\end{equation}
This definition is different then requiring that $u=$$H(y_{1},y_{2},\cdots,y_{N})=H(y)$
be an even or an odd function in each individual coordinate $y_{j}$
. In order to bring out the difference we add the following.
\begin{defn}
Denote by $REG$ an open connected set in $\mathbb{R}^{N}$ . We say
that $H(y)$ is an even function of $y$ in $REG$ in the strict sense
if
\begin{equation}
H(y_{1},y_{2},\cdots,y_{j},\cdots,y_{N})=H(y_{1},y_{2},\cdots,-y_{j},\cdots,y_{N}),\:y\in REG,\:j=1,2,\cdots,N.\label{eq:EVEN(VECTORy)-1}
\end{equation}

We say that $H(y)$ is an odd function of $y$ in $REG$ in the stricter
sense if

\begin{equation}
-H(y_{1},y_{2},\cdots,y_{j},\cdots,y_{N})=H(y_{1},y_{2},\cdots,-y_{j},\cdots,y_{N}),\:y\in REG,\:j=1,2,\cdots,N.\label{eq:ODD(VECTORy)-1}
\end{equation}
\end{defn}

Consider the following functions 
\begin{equation}
H(y_{1},y_{2}):=y_{1}^{5}y_{2}^{3},\:L(y_{1},y_{2})=y_{1}^{10}y_{2}^{6}.\label{eq:EXAMPLE1MULTIVARIABLES}
\end{equation}
 Evidently, $H(y_{1},y_{2}):=y_{1}^{5}y_{2}^{3}$ is an even function
in $REG:=\mathbb{R}^{2}$ . However, it is an odd function in the
strict sense in $REG:=\mathbb{R}^{2}$. Evidently, $L(y_{1},y_{2})=y_{1}^{10}y_{2}^{6}$
is an even function in $REG:=\mathbb{R}^{2}$ and it is also an even
function in the strict sense in $REG:=\mathbb{R}^{2}$.
\begin{rem}
The reader may want to consider a multinomial in the $(r+w)$ independent
variables $y_{1},y_{2},$$\cdots y_{j}\cdots y_{r},y_{r+1},y_{r+2},\cdots,y_{r+w}$
\begin{equation}
H(y)=[y_{1}^{[2e_{1}+1]}y_{2}^{[2e_{2}+1]}\cdots y_{j}^{[2e_{j}+1]}\cdots y_{r}^{[2e_{r}+1]}][y_{r+1}^{[2c_{1}]}y_{r+2}^{[2c_{2}]}\cdots y_{r+w}^{[2c_{w}]}]\label{eq:MULTINOMIAL}
\end{equation}
where $e_{1},e_{2,}\cdots e_{r},c_{1},c_{2},\cdots c_{w}$$\in\mathbb{N}_{0},\;,r,w\in\mathbb{N}\;$
. Formulation of necessary and sufficient conditions on the powers
occurring in $H(y)$ such that a) $H(y)$ is an even multivariate
function b) $H(y)$ is an even multivariate function in the strict
sense c) $H(y)$ is an odd multivariate function d) $H(y)$ is an
odd multivariate function in the strict sense, could further clarify
the difference between these two types of symmetry.
\end{rem}

Next we formulate an analog to Lemma 1 for multi variate functions.
\begin{lem}
Let $H(y)\in C^{1}(REG)$.

i) Assume that $H(y)=H(-y),$$\:y\in REG$. Then the partial derivatives
\[
\varPsi_{j}(y):=\frac{\partial H(y)}{\partial y_{j}},\:j=1,2,\cdots,N
\]

are odd function in $REG$.

ii) Assume that $-H(y)=H(-y),$$\:y\in REG$. Then the partial derivatives
\[
\varPsi_{j}(y):=\frac{\partial H(y)}{\partial y_{j}},\:j=1,2,\cdots,N
\]

are even functions in $REG$.

iii) Assume $f(y)$ to be a column vector function, $f^{T}(y):=[f_{1}(y),f_{2}(y),\cdots,f_{n}(y)]$
, where $f_{j}(y),\:j=1,2,\cdots,n$ are the scalar component of $f(y)$
such that $f_{j}(y)\in C^{1}(REG)$. Then,
\begin{equation}
f(y)=f(-y),\:y\in REG\Longrightarrow\varPsi(y):=\frac{\partial f(y)}{\partial y_{j}}=-\varPsi(-y):=-\frac{\partial f(-y)}{\partial y_{j}},\:j=1,2,\cdots,N.\label{eq:DERVECTORFUNCAREODDif=00005Bf=00005Deven}
\end{equation}

Moreover,

\begin{equation}
-f(y)=f(-y),\:y\in REG\Longrightarrow\varPsi(y):=\frac{\partial f(y)}{\partial y_{j}}=\varPsi(-y):=\frac{\partial f(-y)}{\partial y_{j}},\:j=1,2,\cdots,N.\label{eq:DERVECFUNCTIONSAREEVENIF=00005Bf=00005DisODD}
\end{equation}
\end{lem}

\subjclass[2000]{We first prove i) and we focus on the quotient below: with $y\in REG$
; with $h\neq0$ ; and with $h$ arbitrarily small.
\begin{equation}
Q_{j}(y,h):=\frac{H(y_{1},y_{2},\cdots,y_{j-1},y_{j}+h,y_{j+1},\cdots,y_{N})-H(y_{1},y_{2},\cdots,y_{j-1},y_{j},y_{j+1},\cdots,y_{N})}{h}.\label{eq:Quotient1-2}
\end{equation}
Put as short hand notation
\begin{equation}
\widehat{H}(y_{j}+h):=H(y_{1},y_{2},\cdots,y_{j-1},y_{j}+h,y_{j+1},\cdots,y_{N}),\label{eq:H(y_j+h)Abbreviation}
\end{equation}
\begin{equation}
\widehat{H}(-y_{j}-h):=H(-y_{1},-y_{2},\cdots,-y_{j-1},-y_{j}-h,-y_{j+1},\cdots,-y_{N}).\label{eq:H(y_j+h)Abbreviation-1}
\end{equation}
}
\begin{proof}
Since $H(y)$ is an even function then
\begin{equation}
H(y_{1},y_{2},\cdots,y_{j-1},y_{j}+h,y_{j+1},\cdots,y_{N})=H(-y_{1},-y_{2},\cdots,-y_{j-1},-y_{j}-h,-y_{j+1},\cdots,-y_{N})\label{eq:QUOTIENTCHNGEDBYSYMMETRY1-2}
\end{equation}
\[
H(y_{1},y_{2},\cdots,y_{j-1},y_{j},y_{j+1},\cdots,y_{N})=H(-y_{1},-y_{2},\cdots,-y_{j-1},-y_{j},-y_{j+1},\cdots,-y_{N})
\]
Substitute from (\ref{eq:QUOTIENTCHNGEDBYSYMMETRY1-2}) and (\ref{eq:H(y_j+h)Abbreviation})
and (\ref{eq:H(y_j+h)Abbreviation-1}) into the right hand side of
(\ref{eq:Quotient1-2}) to obtain 
\begin{equation}
Q_{j}(y,h)=-Q_{j}(-y,-h)=-\frac{\widehat{H}(-y_{j}-h)-\widehat{H}(-y_{j})}{-h}.\label{eq:RELATIONQ=00003D-Q-2}
\end{equation}
 Take the limit as $h\rightarrow0$ in (\ref{eq:RELATIONQ=00003D-Q-2})
and obtain

\begin{equation}
\varPsi(y):=\frac{\partial H(y)}{\partial y_{j}}=lim_{h\rightarrow0}Q_{j}(y,h)=-lim_{h\rightarrow0}Q_{j}(-y,-h)=-\varPsi(-y).
\end{equation}

Next we prove ii). We focus again on the quotient 
\begin{equation}
Q_{j}(y,h):=\frac{H(y_{1},y_{2},\cdots,y_{j-1},y_{j}+h,y_{j+1},\cdots,y_{N})-H(y_{1},y_{2},\cdots,y_{j-1},y_{j},y_{j+1},\cdots,y_{N})}{h}.\label{eq:Quotient1-1-1}
\end{equation}
Since $H(y)$ is an odd function then

\begin{equation}
H(y_{1},y_{2},\cdots,y_{j-1},y_{j}+h,y_{j+1},\cdots,y_{N})=\label{eq:QUOTIENTCHNGEDBYSYMMETRY1-2-1}
\end{equation}
\[
-H(-y_{1},-y_{2},\cdots,-y_{j-1},-y_{j}-h,-y_{j+1},\cdots,-y_{N}).
\]
\[
H(y_{1},y_{2},\cdots,y_{j-1},y_{j},y_{j+1},\cdots,y_{N})=-H(-y_{1},-y_{2},\cdots,-y_{j-1},-y_{j},-y_{j+1},\cdots,-y_{N}).
\]
Substitute from (\ref{eq:QUOTIENTCHNGEDBYSYMMETRY1-2-1}) into the
right hand side of (\ref{eq:Quotient1-1-1}) to obtain 
\begin{equation}
Q_{j}(y,h)=Q_{j}(-y,-h).\label{eq:RELATIONQ=00003D-Q-1-1}
\end{equation}
Take the limit as $h\rightarrow0$ in (\ref{eq:RELATIONQ=00003D-Q-1-1})
and obtain

\begin{equation}
\varPsi(y):=\frac{\partial H(y)}{\partial y_{j}}=lim_{h\rightarrow0}Q_{j}(y,h)=lim_{h\rightarrow0}Q_{j}(-y,-h)=\varPsi(-y).\label{eq:ODDDERH}
\end{equation}

The proof of iii) follows from the proofs of i) and ii) and the definition
of $f(y)$.
\end{proof}
We adopt the following notations and conventions under the umbrella
of the following definition.
\begin{defn}
A norm on $\mathbb{R}^{n}$ is denoted by $\left|.\right|$ . The
Euclidean norm is denoted by $\left\Vert .\right\Vert $ . In here
we employ any norm $\left|.\right|$ that also satisfies for $A$
an $m$ by $n$ matrix and $y$ an $n$ by 1 column vector, the following
inequality
\begin{equation}
\left|Ay\right|\leq\left|A\right|\left|y\right|.\label{eq:Extranormrequire}
\end{equation}
Put for a certain continuous function $\phi_{0}(t)\in\mathbb{R}^{n}$

\begin{equation}
D:=\{y\in\mathbb{\mathbb{R}}^{n}\left|y-\phi_{0}(t)\right|\leq b,\;b\geq0\}.\label{eq:HARRYDISK}
\end{equation}

Assume that 
\begin{equation}
\phi_{0}(t)\in C(I);\phi(t)\in C(I);\phi(t)\in D;f(y)\in C(D).\label{eq:ASSUMPTIONSINTEGRALEQ}
\end{equation}
We say that $\phi(t)$ is a solution of the integral equation (\ref{eq:HARRYINTEGRALEQ})
if $\phi(t)$ satisfies the identity relation 
\begin{equation}
\phi(t)\equiv\phi_{0}(t)+\int_{0}^{t}\int_{0}^{u}f(\phi(s))dsdu,\;t\in I.\label{eq:HARRYINTEGRALEQ}
\end{equation}
\end{defn}

The next stage we show how the solution of the initial value problem
(\ref{eq:HARRYSYMMETRICIVP1}) below is equivalent to the solution
of the integral equation (\ref{eq:HARRYINTEGRALEQ}). To this end
we need the following definition.
\begin{defn}
We say that $y=\phi(t)$ is a solution to the initial value problem
(\ref{eq:HARRYSYMMETRICIVP1})

\begin{equation}
y^{\prime\prime}=\frac{d^{2}y}{dt^{2}}=f(y),\;y(0)=y_{0},\thinspace\frac{dy_{j}(0)}{dt}=\eta\label{eq:HARRYSYMMETRICIVP1}
\end{equation}

in $D$ on the interval $I$ if: 
\begin{equation}
y=\phi(t)\in C^{2}(I);y=\phi(t)\in D;f(y)\in C(D)\label{eq:ASSUMPTIONSIVP}
\end{equation}
and $\phi(t)$ satisfies the following identity together with the
initial conditions below. Namely,
\begin{equation}
\phi{}^{\prime\prime}(t)\equiv f(\phi(t)),\,t\in I,\:\phi(0)=y_{0},\phi'(0)=\eta.\label{eq:HARRYINITIALVALUEPROBLEM.}
\end{equation}
\end{defn}

\begin{rem}
We note that at the end points of the interval $I$ the condition
$y=\phi(t)\in C^{2}(I)$ means that the following limits exist
\begin{equation}
\phi'(L):=lim_{h\rightarrow0^{+}}\frac{\phi(L-h)-\phi(L)}{h},\:\phi'(-L):=lim_{h\rightarrow0^{+}}\frac{\phi(-L+h)-\phi(-L)}{h},\label{eq:HARRYDERonesidedlimits}
\end{equation}
\begin{equation}
\phi''(L):=lim_{h\rightarrow0^{+}}\frac{\phi'(L-h)-\phi'(L)}{h},\:\phi''(-L):=lim_{h\rightarrow0^{+}}\frac{\phi'(-L+h)-\phi'(-L)}{h}.\label{eq:HARRY2NDORDERDERonesidedLIMITS}
\end{equation}
\end{rem}

Next we prove
\begin{lem}
Assume: 
\begin{equation}
I=[-L,L];\:y=\phi(t)\in C^{2}(I);\;y=\phi(t)\in D;\;\phi_{0}(t)=y_{0}+t\eta.\label{eq:CONDITIONSLEMMA10}
\end{equation}
 Then, $\phi(t)$ is a solution of (\ref{eq:HARRYSYMMETRICIVP1})
iff it is a solution of (\ref{eq:HARRYINTEGRALEQ}).
\end{lem}

\begin{proof}
Assume $y=\phi(t)\in C^{2}(I)$ is a solution to the initial value
problem (\ref{eq:HARRYSYMMETRICIVP1}) . Then, we may integrate both
sides of (\ref{eq:HARRYSYMMETRICIVP1}) to yield
\begin{equation}
\phi^{\prime}(t)=\eta+\int_{0}^{t}f(\phi(s))ds.\label{eq:HARRYDEREQUATION}
\end{equation}
 An integration of (\ref{eq:HARRYDEREQUATION}) yields 
\begin{equation}
\phi(t)-\phi_{0}(t)=\int_{0}^{t}\int_{0}^{u}f(\phi(s))dsdu.\label{eq:HarrySYMMETRICINTEGRALEQ}
\end{equation}

Thus, $\phi(t)$ is a solution of (\ref{eq:HARRYINTEGRALEQ}). We
proceed with the proof of the converse statement. If $\phi(t)\in C(I)\quad and\quad\phi(t)\in D$
, then $f(\phi(s))$ is well defined on $I$, $f(\phi(s))\in C(I)$
as a composition of continuous functions on $I$ and $\int_{0}^{t}\int_{0}^{u}f(\phi(s))dsdu\in C^{2}(I)$.
Consequently, the right hand side of (\ref{eq:HARRYINTEGRALEQ}) is
twice continuously differentiable on the interval $I$. Consequently,
$\phi(t)$ on the left hand side of (\ref{eq:HARRYINTEGRALEQ}), is
twice continuously differentiable on $I$ . Hence, $\phi(0)=y_{0}$
and $\phi'(0)=\eta$ and $\phi(t)$ is a solution of (\ref{eq:HARRYSYMMETRICIVP1})
as desired yield (\ref{eq:HARRYSYMMETRICIVP1}) with the desired properties.
\end{proof}

\section{Symmetric solutions via successive approximations}

We are ready to formulate and prove the main theorem.
\begin{thm}
Assume that: 
\begin{equation}
t\in I;\:y_{0},\eta,y,f(y)\in\mathbb{\mathbb{R}}^{n},\:n\in\mathbb{N},\phi_{0}(t)=y_{0}+t\eta,\:f(y)\in D.\label{eq:CONDITIONSTHOREM11}
\end{equation}
Assume that $f(y)$ satisfies the Lipchitz condition in $D$. Namely,
there is a constant $K\geq0$ such that for any $\widehat{y},\widetilde{y}\in D$
the following holds

\begin{equation}
\left|f(\widehat{y})-f(\widetilde{y})\right|\leq K\left|\widehat{y}-\widetilde{y}\right|.\label{eq:HARRYLIPSHITZCONDITION}
\end{equation}
Assume also that 
\begin{equation}
D=\{y\bigl|\left\Vert y-\phi_{0}(t)\right\Vert \leq b\}\Longrightarrow\left\Vert f(y)\right\Vert \leq M.\label{eq:Boundedf(y)in a disk.-1}
\end{equation}
 Then, the initial value problem
\begin{equation}
y''=f(y),\;y(0)=y_{0},y'(0)=\eta,\label{eq:IVPY0ETAGENERAL}
\end{equation}

possesses a unique solution $y(t)$ for $\left|t\right|\leq\sqrt{\frac{2b}{M}}$.

The initial value problem

\begin{equation}
y''=f(y),\;y(0)=y_{0},y'(0)=\overrightarrow{0},\;\overrightarrow{0}^{T}:=[0,0,\cdots,0],\label{eq:INITIALVALUEPROBGENERAL-1}
\end{equation}

possesses a unique solution $y(t)$ for $\left|t\right|\leq\sqrt{\frac{2b}{M}}$
such that $y(t)\equiv y(-t)$.

If in addition $-f(y)=f(-y)$ holds for $y\in D$ , then the initial
value problem

\begin{equation}
y''=f(y),\;y(0)=\overrightarrow{0},y'(0)=\eta,\;\overrightarrow{0}^{T}:=[0,0,\cdots,0],\label{eq:INITIALVALUEODD}
\end{equation}

possesses a unique solution $y(t)$ for $\left|t\right|\leq\sqrt{\frac{2b}{M}}$
such that $-y(t)\equiv y(-t)$.

Moreover: for any $t_{0}\in\mathbb{R}$ , the initial value problem
\begin{equation}
y''=f(y),\;y(t_{0})=y_{0},y'(t_{0})=\overrightarrow{0},\;\overrightarrow{0}^{T}:=[0,0,\cdots,0],\label{eq:IVPt0EVEN}
\end{equation}

possesses a unique solution $y(t)$ for $\left|t-t_{0}\right|\leq\sqrt{\frac{2b}{M}}$
such that

$y[(t-t_{0})]\equiv y[-(t-t_{0})]$.

If in addition $-f(y)=f(-y)$ holds for $y\in D$ , then for any $t_{0}\in\mathbb{R}$
, the initial value problem
\begin{equation}
y''=f(y),\;y(t_{0})=\overrightarrow{0},y'(t_{0})=\eta,\;\overrightarrow{0}^{T}:=[0,0,\cdots,0],\label{eq:IVPt0ODD}
\end{equation}

possesses a unique solution $y(t)$ for $\left|t-t_{0}\right|\leq\sqrt{\frac{2b}{M}}$
such that

$-y[(t-t_{0})]\equiv y[-(t-t_{0})]$.

Furthermore, assume that $f(y)$ is defined in every disk $D$ in
(\ref{eq:Boundedf(y)in a disk.-1}), and that for all $b>0$ there
exists $M>0$ such that $\left|f(y)\right|\leq M$. Then, all solutions
to all initial values above exist globally on ($-\infty,\infty$).
\end{thm}

\begin{proof}
We utilize the method of successive approximations e.g. \cite{P.F.Hsieh=000026Y. Sibuya}
Ch. 1 and \cite{BRAUERNOHEL} Ch. 3 and we adapt and specialize it
to the differential system $y''=f(y)$. Denote by $\phi_{j}(t)$ the
successive approximations defined by (\ref{eq:SUCCESSIVE APPROXIMATIONS})
below.
\begin{equation}
\phi_{j+1}(t)-\phi_{0}(t)=\int_{0}^{t}\int_{0}^{u}f(\phi_{j}(s))dsdu;\thinspace j=0,1,2,\ldots,.\label{eq:SUCCESSIVE APPROXIMATIONS}
\end{equation}

Assume that 
\begin{equation}
0\leq s\leq u\leq t\leq\sqrt{\frac{2b}{M}}\quad or\quad that\quad-\sqrt{\frac{2b}{M}}\leq t\leq u\leq s\leq0.\label{eq:DODEFINITIONofINTEGRALS}
\end{equation}
We proceed to show by induction the following three properties of
the successive approximations $\phi_{j}(t)\;j=0,1,2,\ldots,$. We
show that:

i) All $\phi_{j}(t)\in D$ . Namely, that
\begin{equation}
\left|\phi_{j}(t)-\phi_{0}(t)\right|\leq b,\quad if\quad t\in I:=[-\sqrt{\frac{2b}{M}},\sqrt{\frac{2b}{M}}].\label{eq:HARRYDOMAINofDEFINITION-1}
\end{equation}

ii) We show that 
\begin{equation}
\phi_{0}(t)\equiv\phi_{0}(-t)\Longrightarrow\phi_{j}(t)\equiv\phi_{j}(-t),\;j=1,2,\ldots,t\in I:=[-\sqrt{\frac{2b}{M}},\sqrt{\frac{2b}{M}}].\label{eq:HARRYEVENPHIJ}
\end{equation}

iii) We show that if $f(y)$ is an odd function, namely
\begin{equation}
-f(y)\equiv f(-y),\;y\in D,\label{eq:HARRYf(y)EVENINRECbox}
\end{equation}

then 
\begin{equation}
-\phi_{0}(t)\equiv\phi_{0}(-t)\Longrightarrow-\phi_{j}(t)\equiv\phi_{j}(-t),\;j=1,2,\ldots,t\in I.\label{eq:HARRYODDPHIJt}
\end{equation}

We proceed with the proofs. Evidently, for $j=0$ 
\begin{equation}
\phi_{0}(t)\in D\quad because\quad\left|\phi_{0}(t)-\phi_{0}(t)\right|\equiv0\leq b.\label{eq:HARRYPHI0INDOD}
\end{equation}
For $\phi_{0}(t)$ all three conditions i), ii) and iii) hold trivially.
Assume by induction that $\phi_{j}(t)$ satisfies (\ref{eq:HARRYDOMAINofDEFINITION-1}).
We desire to prove that
\begin{equation}
\left|\phi_{j+1}(t)-\phi_{0}(t)\right|\leq b,\quad if\quad t\in I:=[-\sqrt{\frac{2b}{M}},\sqrt{\frac{2b}{M}}].\label{eq:HARRYDODforphi(j+1)}
\end{equation}

By definition (\ref{eq:SUCCESSIVE APPROXIMATIONS}) of $\phi_{j+1}$
, we have with $0\leq s\leq u\leq t$ the desired estimate
\begin{equation}
\left|\phi_{j+1}(t)-\phi_{0}(t)\right|\leq\int_{0}^{t}\int_{0}^{u}Mdsdu\leq\frac{1}{2}Mt^{2}\leq\frac{1}{2}M(\frac{2b}{M})=b.\label{eq:HARRYPHIJ+1inRECBOX.}
\end{equation}

If $t\leq u\leq s\leq0$ then substitute in the double integral $I_{j}:=\int_{0}^{t}\int_{0}^{u}f(\phi_{j}(s))dsdu$
\begin{equation}
s=-v_{1}\Longrightarrow\:ds=-dv_{1},\:and\:\{s=0\Longleftrightarrow v_{1}=0\},\:and\:\{s=u\Rightarrow v_{1}=-u\}\label{eq:HARRY!STsubstituteDoubleIntegral}
\end{equation}

and obtain 
\begin{equation}
I_{j}:=\int_{0}^{t}[\int_{0}^{u}f(\phi_{j}(s))ds]du=\int_{0}^{t}[\int_{0}^{-u}f(\phi_{j}(-v_{1}))(-dv_{1})]du.\label{eq:HARRYCHANGE1VARDI}
\end{equation}

Change again variables in (\ref{eq:HARRYCHANGE1VARDI}) as follows
\begin{equation}
u=-v_{2}\Longrightarrow\:du=-dv_{2},\:and\:\{u=0\Longleftrightarrow v_{2}=0\},\:and\:\{u=t\Rightarrow v_{2}=-t\}.\label{eq:HARRYFIRSTVARIABLESCHANGE}
\end{equation}

Then, 
\begin{equation}
I_{j}:=\int_{0}^{t}[\int_{0}^{u}f(\phi_{j}(s))ds]du=\int_{0}^{t}[\int_{0}^{-u}f(\phi_{j}(-v_{1}))(-dv_{1})]du=\label{eq:HARRY2NDCHANGEVARINIJ}
\end{equation}
\[
\int_{0}^{-t}[\int_{0}^{v_{2}}f(\phi_{j}(-v_{1}))(-dv_{1})](-dv_{2}).
\]

In sum 
\begin{equation}
-\sqrt{\frac{2b}{M}}\leq t\leq u\leq s\leq0\Longleftrightarrow0\leq v_{1}\leq v_{2}\leq-t\leq\sqrt{\frac{2b}{M}}.\label{eq:HARRYsum2VARIABLESChange}
\end{equation}

By virtue of two changes of variables we obtain two non negative upper
limits in (\ref{eq:HARRY2NDCHANGEVARINIJ}) . This permits us to estimate
the norm,
\begin{equation}
\left|I_{j}\right|=\left|\phi_{j+1}(t)-\phi_{0}(t)\right|=\left|\int_{0}^{t}[\int_{0}^{u}f(\phi_{j}(s))ds]du\right|\leq\int_{0}^{-t}[\int_{0}^{v_{2}}\left|f(\phi_{j}(-v_{1}))\right|dv_{1}dv_{2}\leq\label{eq:PHIJ+1ESTIMATEFORNEGATIVEt.}
\end{equation}
\[
\leq\int_{0}^{-t}\int_{0}^{v_{2}}Mdv_{1}dv_{2}\leq\frac{1}{2}M(-t)^{2}\leq\frac{1}{2}M(\frac{2b}{M})=b,
\]

as desired. In order to complete the proof for ii) and iii) above,
we observe that if $\phi(s)=\phi(-s)$ for $s\in I$ and $\phi(s)\in D$
then $f(\phi(s))=f(\phi(-s))$ . Hence $f(\phi(s))$ is also an even
function of $s$ for $s\in I$. If

$\phi_{0}(s)=\phi_{0}(-s)$ for $s\in I$ then by Lemma 2, all functions
in (\ref{eq:HARRYDOBLEINTEGRALandPHIJEVEN}) are even functions for
$s\in I$ as desired. Namely,
\begin{equation}
\int_{0}^{t}[\int_{0}^{u}f(\phi_{0}(s))ds]du,\;\int_{0}^{t}[\int_{0}^{u}f(\phi_{j}(s))ds]du,\label{eq:HARRYDOBLEINTEGRALandPHIJEVEN}
\end{equation}
\[
\phi_{j}(s):=\phi_{0}(s)+\int_{0}^{t}[\int_{0}^{u}f(\phi_{j-1}(s))ds]du,\;j=1,2,\ldots,.
\]

We turn to successive approximations that are odd functions. Notice
that 
\begin{equation}
-\phi(s)=\phi(-s)\quad s\in I,\quad-f(y)=f(-y),\;\phi(s),y\in D\Longrightarrow\label{eq:HARRYDOUBLEINTEGRALSPKIODD}
\end{equation}
\[
-f(\phi(s))=-f(-\phi(-s))=f(\phi(-s)).
\]

Therefore, if $-\phi_{0}(s)=\phi_{0}(-s)$ for $s\in I$ , then by
Lemma 2, all expressions in (\ref{eq:HARRYDOUBLEINTEGRALSPKIODD})
are odd functions for $s\in I$ as desired. Namely, 
\begin{equation}
\int_{0}^{t}[\int_{0}^{u}f(\phi_{0}(s))ds]du,\;\int_{0}^{t}[\int_{0}^{u}f(\phi_{j}(s))ds]du,\label{eq:LISTofEntities}
\end{equation}
\begin{equation}
\phi_{j}(s):=\phi_{0}(s)+\int_{0}^{t}[\int_{0}^{u}f(\phi_{j-1}(s))ds]du,\;j=1,2,\ldots,.\label{eq:LISTofENTITIES1}
\end{equation}

In the upcoming discussion we will show that the successive approximations
$\phi_{j}(t)$ , converge absolutely and uniformly to a continuous
function $\phi(t)$ on the interval $I$.

The method of proof requires special consideration of $\phi_{0}(t)$
and of $\phi_{1}(t)-\phi_{0}(t)$ . The rest of the differences $\phi_{j+1}(t)-\phi_{j}(t)\thinspace;j=1,2,\ldots,$
follow similar patterns. We assume that in the double integrals below
the upper bounds in the integrals obey the order $0\leq s\leq u\leq t$$\leq L.$
Recall that the successive approximations are defined so that
\begin{equation}
\left|\phi_{1}(t)-\phi_{0}(t)\right|\leq\int_{0}^{t}\int_{0}^{u}\left|f(\phi_{0}(s))\right|dsdu\Longrightarrow\left|\phi_{1}(t)-\phi_{0}(t)\right|\leq\int_{0}^{t}\int_{0}^{u}Mdsdu=\frac{1}{2!}Mt^{2},\label{eq:Harry1Even}
\end{equation}
\begin{equation}
\phi_{2}(t)-\phi_{1}(t)=\int_{0}^{t}\int_{0}^{u}[f(\phi_{1}(s))-f(\phi_{0}(s)]dsdu\Longrightarrow\label{eq:HARRY2EVEN}
\end{equation}
\[
\left|\phi_{2}(t)-\phi_{1}(t)\right|\leq\int_{0}^{t}\int_{0}^{u}\left|f(\phi_{1}(s))-f(\phi_{0}(s)\right|dsdu.
\]

Insert the Lipchitz condition (\ref{eq:HARRYLIPSHITZCONDITION}) in
(\ref{eq:HARRY2EVEN}) and obtain 
\begin{equation}
\left|\phi_{2}(t)-\phi_{1}(t)\right|\leq\int_{0}^{t}\int_{0}^{u}K\left|\phi_{1}(s)-\phi_{0}(s)\right|dsdu.\label{eq:HARRY3EVEN}
\end{equation}

Now use the special estimate on the right hand side of (\ref{eq:Harry1Even})
to obtain 
\begin{equation}
\left|\phi_{2}(t)-\phi_{1}(t)\right|\leq\int_{0}^{t}\int_{0}^{u}K[\frac{1}{2!}M]s^{2}dsdu=\frac{M}{4!}Kt^{4}=\frac{M}{4!K}(K)^{2}t^{4}.\label{eq:HARRY4EVEN}
\end{equation}

The reader can easily verify that 
\begin{equation}
\left|\phi_{3}(t)-\phi_{2}(t)\right|\leq\int_{0}^{t}\int_{0}^{u}K[\frac{1}{4!}KMs^{4}]dsdu=\frac{M}{6!K}(K)^{3}t^{6}.\label{eq:HARRY5EVEN}
\end{equation}
The formulas (\ref{eq:HARRY4EVEN}) and (\ref{eq:HARRY5EVEN}) make
it is easy to guess an induction pattern that is 
\begin{equation}
\ref{eq:HARRYEVENINDUCTIONPATTERN}\left|\phi_{j}(t)-\phi_{j-1}(t)\right|\leq\frac{M}{(2j)!K}(K)^{j}t^{2j},\;j=1,2,\ldots,.\label{eq:HARRYEVENINDUCTIONPATTERN}
\end{equation}

Without loss of generality we may assume that the expressions $\frac{M}{(2j)!K}(K)^{j}t^{2j}$
are well defined for $K=0$ as well. This makes $f(y)$ by virtue
of (\ref{eq:HARRYLIPSHITZCONDITION}) a constant vector and the solution
$\phi(t)\equiv\phi_{0}(t)$.

In order to aver the induction pattern we need to show that (\ref{eq:HARRYEVENINDUCTIONPATTERN})
implies
\begin{equation}
\left|\phi_{j+1}(t)-\phi_{j}(t)\right|\leq\frac{M}{[2(j+1)]!K}(K)^{j+1}t^{2(j+1)},\;j=1,2,\ldots,.\label{eq:HARRYEVEN(j+1)INDUCTION}
\end{equation}

Indeed, insert the induction inequality (\ref{eq:HARRYEVENINDUCTIONPATTERN})
into 
\begin{equation}
\left|\phi_{j+1}(t)-\phi_{j}(t)\right|\leq\int_{0}^{t}\int_{0}^{u}K\left|\phi_{j}(s)-\phi_{j-1}(s)\right|dsdu\leq\label{eq:HARRYINDUCTIONJ2(J+1)}
\end{equation}
\[
\int_{0}^{t}\int_{0}^{u}K[\frac{M}{(2j)!K}(K)^{j}s^{2j}]dsdu.
\]

and obtain the desired conclusion (\ref{eq:HARRYEVEN(j+1)INDUCTION}).
It is easily verified by the ratio test that the majorant series $\sum_{j=1}^{\infty}\frac{M}{(2j)!K}(K)^{j}t^{2j}$
are absolutely and uniformly convergent on any closed interval $[-r,r],\;r>0$.
This is verified by the ratio test 
\begin{equation}
lim_{j\rightarrow\infty}\frac{\frac{M}{[2(j+1)]!K}(K)^{j+1}t^{2(j+1)}}{1}\frac{1}{\frac{M}{(2j)!K}(K)^{j}t^{2j}}=lim_{j\rightarrow\infty}\frac{Kt^{2}}{(2j+2)(2j+1)}=0.\label{eq:HARRYSYMMETRYRATIOTEST}
\end{equation}

The series define a function $\psi(t)$ such that for $t\in I=[-\sqrt{\frac{2b}{M}},\sqrt{\frac{2b}{M}}]$
we have 
\begin{equation}
\psi(t):=\sum_{j=1}^{\infty}\frac{M}{(2j)!K}(K)^{j}t^{2j},\label{eq:HARRYMAJORANTSERIES}
\end{equation}

\begin{equation}
\sum_{j=1}^{\infty}\left|\phi_{j}(t)-\phi_{j-1}(t)\right|\leq\psi(t):=\sum_{j=1}^{\infty}\frac{M}{(2j)!K}(K)^{j}t^{2j}\leq\psi(\sqrt{\frac{2b}{M}}).\label{eq:SERIES|PHIJ+!-PHIJ|UNIFORMLY}
\end{equation}

The relation (\ref{eq:SERIES|PHIJ+!-PHIJ|UNIFORMLY}) make the series
\begin{equation}
\phi_{0}(t)+\sum_{j=1}^{\infty}[\phi_{j}(t)-\phi_{j-1}(t)],\label{eq:INFINITESUM=00005BPHIJ+1-PHIJ=00005D}
\end{equation}

absolutely and uniformly convergent for $t\in I=[-\sqrt{\frac{2b}{M}},\sqrt{\frac{2b}{M}}]$.
Hence, a function $\phi(t)$ on $[-\sqrt{\frac{2b}{M}},\sqrt{\frac{2b}{M}}]$
exists such that 
\begin{equation}
\phi(t):=lim_{j\rightarrow\infty}\phi_{j+1}(t)=\phi_{0}(t)+lim_{j\rightarrow\infty}\int_{0}^{t}\int_{0}^{u}f(\phi_{j}(s))dsdu=\label{eq:HARRYSOLUTION@INTEGRALEQUATION}
\end{equation}
\[
\phi_{0}(t)+\int_{0}^{t}\int_{0}^{u}lim_{j\rightarrow\infty}f(\phi_{j}(s))dsdu=
\]
\[
\phi_{0}(t)+\int_{0}^{t}\int_{0}^{u}f(lim_{j\rightarrow\infty}(\phi_{j}(s))dsdu=
\]
\[
\phi_{0}(t)+\int_{0}^{t}\int_{0}^{u}f((\phi(s))dsdu.
\]

If we choose $\phi_{0}(t)\equiv y_{0}$ then the solution $\phi(t):=lim_{j\rightarrow\infty}\phi_{j}(t)$
will be an even solution. If we choose $\phi_{0}(t)\equiv t\eta$
then $\phi(t):=lim_{j\rightarrow\infty}\phi_{j}(t)$ will be an on
odd function on condition that $-f(y)=f(-y)$.

In order to obtain solutions to (\ref{eq:IVPt0EVEN}) and (\ref{eq:IVPt0ODD})
, together with the estimates $\left|t-t_{0}\right|\leq\sqrt{\frac{2b}{M}}$
on the interval of existence, we consider the initial value problem
\begin{equation}
\frac{d^{2}\widehat{y}}{d\tau^{2}}=f(\widehat{y}),\;\widehat{y}(0)=y_{0},\frac{d\widehat{y}}{d\tau}(0)=\eta,\label{eq:TAUNEWIVP}
\end{equation}

where $\tau$ is a new independent variable and $\widehat{y}(\tau)$
is the unique solution of (\ref{eq:TAUNEWIVP}). Theorem 11 applies
to the existence of solutions $\widehat{y}(\tau)$ with the desired
properties of symmetry. Observe that we have 
\begin{equation}
\tau=t-t_{0}\Longrightarrow\{\tau=0\Longleftrightarrow t=t_{0}\}.\label{eq:INITIALTAUINTIALt0}
\end{equation}

Put 
\begin{equation}
y(t):=\widehat{y}(\tau)=\widehat{y}(t-t_{0})\Longrightarrow y(t_{0}):=\widehat{y}(0),\;\frac{dy(t_{0})}{dt}=\frac{d\widehat{y}(0)}{d\tau}.\label{eq:INITIALCONDITIONSTAUIMPLYICt}
\end{equation}

It is now easily verified that 
\begin{equation}
\frac{dy}{dt}=\frac{d\widehat{y}}{d\tau}\frac{d\tau}{dt}=\frac{d\widehat{y}}{d\tau}\Longrightarrow\frac{d^{2}y}{dt^{2}}=\frac{d^{2}\widehat{y}}{d\tau^{2}}.\label{eq:INVARIANTDERIVATIVES}
\end{equation}

Since the vector function of $f(y)$ is independent of $t$ we conclude
that the existence of solutions to the initial value problem (\ref{eq:TAUNEWIVP}),
implies the existence of corresponding solutions to the initial value
problem (\ref{eq:INPINtgeneral})
\begin{equation}
\frac{d^{2}y}{dt^{2}}=y''=f(y),\;y(0)=y_{0},y'(0)=\eta.\label{eq:INPINtgeneral}
\end{equation}

Conversely, the existence of solutions to the initial value problem
(\ref{eq:INPINtgeneral}) , implies the existence of corresponding
solutions to the initial value problem (\ref{eq:TAUNEWIVP}).

Last but not least, the existence of global solutions to (\ref{eq:INPINtgeneral})
under the conditions stated follows by letting $b\rightarrow\infty$
in (\ref{eq:DODEFINITIONofINTEGRALS}).
\end{proof}
\begin{rem}
It is easily verified that the Lipchitz condition (\ref{eq:HARRYLIPSHITZCONDITION})
renders the function $f(y)$ continuous in $D$.
\end{rem}

Notice also
\begin{rem}
If $f(y)$ is defined for all $y\in\mathbb{R}^{n}$ and satisfies
a sub-linearity condition 
\begin{equation}
\left|f(y)\right|\leq M_{1}\left|y\right|+M_{2},\;M_{1},M_{2}\geq0,\label{eq:Sublinearitycondition.}
\end{equation}
and $M_{1},M_{2}$ are certain constants independent of $y$ , then
the Gronwall Lemma e.g. \cite{BRAUERNOHEL} Chapter 1, implies that
there exists a constant $M_{3}>0$ for a given initial value problem,
such that 
\begin{equation}
\left|f(y)\right|\leq M,\;M:=M_{1}M_{3}+M_{2},\;M_{1},M_{2},M_{3}\geq0.\label{eq:GronwallRenders|f(y)|lessM}
\end{equation}
The condition (\ref{eq:Sublinearitycondition.}) applies to the the
pendulum equation $y''=-sin(y)$ with $M_{1}=0$ and $M=M_{2}=1$
for all initial value problems.
\end{rem}

Consider also
\begin{rem}
How are symmetric solutions to first order systems $z'=H(z)$ scarce?
In order to better understand this scarcity, consider an initial value
problem 
\begin{equation}
z'=H(z),\,z(t_{0})=z_{0},\qquad,z,H(z)\in\mathbb{R}^{m},\;m\in\mathbb{N}.\label{eq:FIRSTORDERSYSTEMS}
\end{equation}

Assume that $m\neq2n$ or assume in case that $m=2n$ that $z'=H(z)$
does not originate from $z^{T}=(y,y')$ where $H(z)^{T}=(y',f(y)$).
Also assume that $H(z)\in C^{1}(D)$ where $D$ is some region of
$\mathbb{R}^{m}$. If (\ref{eq:FIRSTORDERSYSTEMS}) possesses an even
solution then $z'(t_{0})=\overrightarrow{0}=H(z_{0})$ . Thus, $z(t)\equiv z_{0}$
must be the unique constant even solution. If (\ref{eq:FIRSTORDERSYSTEMS})
possesses an odd non constant vector solution then the initial value
problem (\ref{eq:FIRSTORDERSYSTEMS}) must obey $z_{0}=\overrightarrow{0}$.
Denote by $JH(z)$ the Jacobean of $H(z)$. Then, we must have $z''=[JH(z(t_{0}))]z'(t_{0})=[JH(\overrightarrow{0})]H(\overrightarrow{0})=\overrightarrow{0}$.
Thus the Jacobean evaluated at $\overrightarrow{0}$ must have an
eigenvalue zero corresponding to the eigenvector $H(\overrightarrow{0})$
if $H(\overrightarrow{0})$$\neq$$\overrightarrow{0}$. The requirement
that some or all even derivatives $z^{(2k)}(0)=\overrightarrow{0},k=1,2,\ldots,$
impose further restrictions on $H(z)$ . It is in this sense that
first order systems normally do not possess non constant odd solutions.
\end{rem}

\section{examples of $y''=f(y)$ occurring in applications}

In this section we provide examples of second order differential systems
and scalar equations to which Theorem 11 applies. These nonlinear
differential systems and equations are autonomous and $f(y)$ is independent
of $y'$ . Normally, these are conservative non dissipative systems
or systems without damping.

Noteworthy is the N body problem. Adopt the following notation: $\thinspace m_{1},m_{2},\ldots,m_{N}$
are the masses of the $N$ bodies; $t\in\mathbb{R}$ is the time variable;
$y_{j}\in\mathbb{R}^{3}$ and $1\leq j\leq N$, are column position
arrows of the $N$ bodies, respectively; $T$ stands for transposition
of a vector or a matrix; $y^{T}=[y_{1,}y_{2,},\cdots,y_{N}]$ and
$f(y)^{T}:=[f_{1}(y),f_{2}(y),\cdots,f_{N}(y)]^{T}$ are respectively
rows of blocks of column vectors with 
\begin{equation}
y_{k}^{''}=f_{k}(y):=\sum_{j\ne k}\frac{Gm_{j}(y_{j}-y_{k})}{\|y_{j}-y_{k}\|^{3}}=\frac{1}{m_{k}}\nabla_{y_{k}}U,\;U:=\sum_{j<k}\frac{m_{j}m_{k}}{\|y_{j}-y_{k}\|}>0.\label{eq:NEWTON'SEQUATIONS}
\end{equation}
 $U$ is the gravitational potential; $\nabla_{y_{k}}U$ is the gradient
of $U$ with respect to the components of $y_{k}$; $\left\Vert y\right\Vert $
is the Euclidean norm , $G$ is the gravitational constant.

Consider the initial value problem for the $N$ body problem in celestial
mechanics

\begin{equation}
y''=f(y),\;y(t_{0})=y_{0},y'(t_{0})=\eta,\;y_{k}(t_{0})\neq y_{j}(t_{0}),\;k\neq j,\;k,j=1,2,\cdots,N.\label{eq:NBODY2NDORDERVEC-1-1}
\end{equation}

Compare with \cite{POLARDCELESTIAL}{\small{}. Choosing the velocity
vector $\eta=\overrightarrow{0}$ , then by Theorem 11 , the initial
value problem 
\begin{equation}
y''=f(y),\;y(t_{0})=y_{0},y'(t_{0})=\overrightarrow{0},\;y_{k}(t_{0})\neq y_{j}(t_{0}),\;k\neq j,\;k,j=1,2,\cdots,N,\label{eq:ZEROINITIALVELOCITIES1}
\end{equation}
}{\small\par}

{\small{}possesses a continuum of even solutions by varying the position
vector $y_{0}$ according to (\ref{eq:ZEROINITIALVELOCITIES1}). Observe
that Newtons equations of celestial mechanics satisfy 
\begin{equation}
-f_{k}(-y)=-\sum_{j\ne k}\frac{Gm_{j}(-y_{j}+y_{k})}{\|y_{j}-y_{k}\|^{3}}=\sum_{j\ne k}\frac{Gm_{j}(y_{j}-y_{k})}{\|y_{j}-y_{k}\|^{3}}\Longrightarrow-f(y)=f(-y).\label{eq:NEWTON-f(y)=00003Df(-y).}
\end{equation}
However, we cannot solve by Theorem 11 for odd solutions an initial
value problem with a condition $y(t_{0})=\overrightarrow{0}$ . This,
because $y(t_{0})=\overrightarrow{0}$ means that all of the celestial
point masses are in state of mutual collision. Then, the equations
in (\ref{eq:NEWTON'SEQUATIONS}) contain undetermined and unbounded
terms which render the equations invalid.}{\small\par}

{\small{}Past, Present and future of the planetary motion and the
Universe are of great interest. }Compare with \cite{Brumberg}. Choose
in (\ref{eq:NBODY2NDORDERVEC-1-1}) $\eta=\overrightarrow{0}$ and
obtain even solutions about $t_{0}$ where $y[(t-t_{0})]\equiv y[-(t-t_{0})]$.
This means that with initial velocities zero, the trajectories of
the $N$ point masses in the future $t>t_{0}$ are a perfect reflection
of the past $t<t_{0}$.

We selected from the textbooks \cite{Boyce E.W. and DiPrima R.C,BRAUERNOHEL,GuckenheimerHolmes,JORDAN=000026SMITH,L. Perko,P.F.Hsieh=000026Y. Sibuya,POLARDCELESTIAL}
and from certain journals, more examples. A substantial number of
differential equations are obtained from a certain potentials $U(y)$
as follows

\begin{equation}
y''=f(y),with\quad y\in\mathbb{\mathbb{R}}^{3},f(y)=-\nabla U(y).\label{eq:POTENTIALandf(y)}
\end{equation}
In a series of papers published between 1924 and 1930, G. Manev studied
a modification to the Newtonian potential $U(y):=-\frac{\gamma}{\left\Vert y\right\Vert }$
where $\gamma=G(m_{1}+m_{2})$ . With $\epsilon>0$ being a certain
parameter Manev studied the potential.

\begin{equation}
U(y):=-\frac{\gamma}{\left\Vert y\right\Vert }-\frac{\epsilon}{\left\Vert y\right\Vert ^{2}}.\label{eq:Manev}
\end{equation}

For a detailed discussion of the history and applications of the Manev
potential we refer the reader to \cite{MANEVI,MANEVDIACUII}. Another
differential system $y''=f(y)$ originates from the Kepler anisotropic
potential. 
\[
U(y)==\text{\textminus}\frac{1}{\sqrt{y_{1}^{2}+y_{2}^{2}}}-\frac{b}{(\mu y_{1}^{2}+y_{2}^{2})^{\frac{\beta}{2}}},\quad\beta\geq2,\;\mu\geq1,\;b>0.
\]
\smallskip{}

Compare with \cite{DIACUISENTROPIC}. The list of equations of scalar
nonlinear equations $x''=f(x)$ below is taken from \cite{JORDAN=000026SMITH}
. Notice that in the first six differential equations listed below,
$f(x)$ is an odd function of $x$. Consequently, the first six equations
possess a continuum of even solutions and a continuum of odd solutions.

In what follows $x,\alpha,a,\omega,\theta,\lambda$,$g,u,m,F,h,\phi,c$,
are real values. A few comments may accompany the equations below.

The equation

\begin{equation}
\frac{d^{2}x}{dt^{2}}=-x-\alpha x^{3}.
\end{equation}

The equation

\begin{equation}
\frac{d^{2}x}{dt^{2}}=-9x,
\end{equation}

Two stars, each with gravitational mass $\mu$, are orbiting each
other under their mutual gravitational forces in such away that their
orbits are circles of radius $a$. A satellite of relatively negligible
mass is moving on a straight line through the mass center $G$ such
that the line is perpendicular to the plane of the mutual orbits of
this binary system. The equation of motion is then given by
\begin{equation}
\frac{d^{2}x}{dt^{2}}=-\frac{2\mu x}{(a^{2}+x^{2})^{\frac{3}{2}}}.
\end{equation}

The Pendulum equation is given by
\[
\frac{d^{2}x}{dt^{2}}=-\omega^{2}\sin(x).
\]
 It may be approximated for moderate amplitudes by the equation 
\begin{equation}
\frac{d^{2}x}{dt^{2}}=-\omega^{2}(x-\frac{1}{6}x^{3}),
\end{equation}

where $x$ is the inclination.

A Pendulum of length $a$ has a bob of mass $m$ which is subject
to a horizontal force $m\omega^{2}a\sin\theta$, where $\theta$ is
the inclination to the downward vertical. The mathematical model governs
the motion is given by: 
\begin{equation}
\frac{d^{2}\theta}{dt^{2}}=\omega^{2}(\cos\theta-\lambda)\sin\theta.
\end{equation}

A particle is attached to a fixed point $O$ on a smooth horizontal
plane by an elastic string, $x$ being the displacement from $O$.
When unstreched, the length of the string is $2a$. The equation of
motion of the particle, which is constrained to move on a straight
line through $O$, is given by: 
\begin{equation}
\frac{d^{2}x}{dt^{2}}=-x+a\thinspace sgn(x),\quad if\quad\mid x\mid>a,\quad\frac{d^{2}x}{dt^{2}}=0,\quad if\quad\mid x\mid\leq a.
\end{equation}

The equation

\begin{equation}
\frac{d^{2}x}{dt^{2}}=x^{4}-x^{2}.
\end{equation}

The equation

\begin{equation}
\frac{d^{2}x}{dt^{2}}=(x-\lambda)(x^{2}-\lambda).
\end{equation}

A Pendulum with a magnetic bob oscillates in a vertical plane over
a magnet, which repels the bob according to the inverse square law,
the equation of motion is then given by: 
\begin{equation}
ma^{2}\frac{d^{2}\theta}{dt^{2}}=-mga\sin\theta+Fh\sin\phi,
\end{equation}
 where $h>a$ and 
\begin{equation}
F=\frac{c}{a^{2}+h^{2}-2ah\cos\theta}.
\end{equation}

The equation 
\begin{equation}
\frac{d^{2}x}{dt^{2}}=-\lambda-x^{3}+x.
\end{equation}

The equation 
\begin{equation}
\frac{d^{2}x}{dt^{2}}=a-\exp(x).
\end{equation}

The equation 
\begin{equation}
\frac{d^{2}x}{dt^{2}}=a+\exp(x).
\end{equation}

The equation of motion of a conservative system is in more general
form studied in various text books. See also \cite{BRAUERNOHEL}.
\begin{equation}
\frac{d^{2}x}{dt^{2}}=-g(x),
\end{equation}
 where $g(0)=0$, $g(x)$ is strictly increasing for all $x$, and
\begin{equation}
\int_{0}^{x}g(u)du\rightarrow\infty\:as\thinspace x\rightarrow\pm\infty.
\end{equation}

\end{document}